\documentclass[12pt]{amsart}
\usepackage{amsmath}
\usepackage{amsfonts}
\usepackage{amssymb}
\begin{document}
\title[A $q$-analogue of derivations on the tensor algebra]
{A $q$-analogue of derivations on the tensor algebra and
the $q$-Schur--Weyl duality}
\author{Minoru ITOH}
\date{}
\address{Department of Mathematics and Computer Science, 
          Faculty of Science,
          Kagoshima University, Kagoshima 890-0065, Japan}
\email{itoh@sci.kagoshima-u.ac.jp }
\thanks{This research was partially supported by 
JSPS Grant-in-Aid for Young Scientists (B) 24740021.}
\begin{abstract}
This paper presents a $q$-analogue of an extension of the tensor algebra given by the same author.
This new algebra naturally contains the ordinary tensor algebra and
the Iwahori--Hecke algebra type $A$ of infinite degree.
Namely this algebra can be regarded as a natural mix of these two algebras.
Moreover, we can consider natural ``derivations'' on this algebra.
Using these derivations, we can easily prove the $q$-Schur--Weyl duality 
(the duality between the quantum enveloping algebra of the general linear Lie algebra 
and the Iwahori--Hecke algebra of type $A$).
\end{abstract}
\maketitle
\theoremstyle{plain}
   \newtheorem{theorem}{Theorem}[section]
   \newtheorem{proposition}[theorem]{Proposition}
   \newtheorem{lemma}[theorem]{Lemma}
   \newtheorem{corollary}[theorem]{Corollary}
\theoremstyle{remark}
   \newtheorem*{remark}{Remark}
   \newtheorem*{remarks}{Remarks}
\numberwithin{equation}{section}
\newcommand{\mybinom}[2]{\left(\!\genfrac{}{}{0pt}{}{#1}{#2}\!\right)}
\newcommand{\bibinom}[2]{\left(\!\!\left(\!\genfrac{}{}{0pt}{}{#1}{#2}\!\right)\!\!\right)}
%
%
\section{Introduction}
%
%
This paper presents a $q$-analogue of an extension of the tensor algebra given in [I].
Using this algebra, we can easily prove the $q$-Schur--Weyl duality 
(the duality between the quantum enveloping algebra $U_q(\mathfrak{gl}_n)$ 
and the Iwahori--Hecke algebra of type $A$).

First let us recall the algebra $\bar{T}(V)$ given in [I].
This algebra $\bar{T}(V)$ naturally 
contains the ordinary tensor algebra $T(V)$
and the infinite symmetric group $S_{\infty}$.
Moreover, we can consider natural ``derivations'' on this algebra,
which satisfy an analogue of canonical commutation relations.
This algebra and these derivations are useful to study representations on the tensor algebra.
For example, we can prove the Schur--Weyl duality easily using this framework.

In this paper, we give a $q$-analogue of this algebra $\bar{T}(V)$.
This new algebra $\hat{T}(V)$ naturally contains the ordinary tensor algebra $T(V)$
and the Iwahori--Hecke algebra $H_{\infty}(q)$ of type $A_{\infty}$.
Namely we can regard this $\hat{T}(V)$ as a natural mix of $T(V)$ and $H_{\infty}(q)$.
We can also consider natural ``derivations'' on the algebra $\hat{T}(V)$.
These derivations are useful to describe the natural action of $U_q(\mathfrak{gl}_n)$ on $V^{\otimes p}$.
Moreover, using these derivations, we can easily prove the $q$-Schur--Weyl duality.

Some applications of $\bar{T}(V)$ were given in \cite{I}:
(i) invariant theory in the tensor algebra 
(for example, a proof of the first fundamental theorem of invariant theory
with respect to the natural action of the special linear group), and
(ii) application to immanants and the quantum immanants (a linear basis of the center of 
the universal enveloping algebra $U(\mathfrak{gl}_n)$; see \cite{O1} and \cite{O2}).
The author hopes that the algebra $\hat{T}(V)$ will be useful to study 
representation theory and invariant theory related to $U_q(\mathfrak{gl}_n)$.

The author would like to thank the referee for the valuable comments. 
%
%
\section{Definition of $\hat{T}(V)$}
%
%
Let us start with the definition of the algebra $\hat{T}(V)$
determined by a vector space $V = \mathbb{C}^n$. 
We recall that the ordinary tensor algebra is defined by
$$
   T(V) = \bigoplus_{p \geq 0} T_p(V)
$$
with $T_p(V) = V^{\otimes p}$.
Noting this, we define $\hat{T}(V)$ as a vector space by
$$
   \hat{T}(V) = \bigoplus_{p \geq 0} \hat{T}_p(V),
$$
where $\hat{T}_p(V)$ is the following induced representation:
$$
   \hat{T}_p(V)
   = \operatorname{Ind}_{H_p(q)}^{H_{\infty}(q)} V^{\otimes p}
   = H_{\infty}(q) \otimes_{H_p(q)} V^{\otimes p}.
$$

Here the notation is as follows.
First $H_p(q)$ is the Iwahori--Hecke algebra of type $A_{p-1}$.
Namely this is the $\mathbb{C}$-algebra defined by the following generators and relations:
\begin{align*}
   \text{generators: }  
   &t_1,\ldots,t_{p-1},  \\
   \text{relations: }  
   &(t_r - q)(t_r + q^{-1}) = 0, \\
   &t_r t_{r+1} t_r = t_{r+1} t_r t_{r+1}, \\
   &t_r t_s = t_s t_r, \quad \text{for $|r-s| > 1$}.
\end{align*}
We define $H_{\infty}(q)$ as the inductive limit of
the natural inclusions $H_0(q) \subset H_1(q) \subset \cdots$.
Next, $H_p(q)$ naturally acts on $T_p(V) = V^{\otimes p}$ as follows (\cite{J}):
$$
   t_r = \overbrace{\operatorname{id}_V \otimes \cdots \otimes \operatorname{id}_V}^{r-1}
    \otimes t \otimes 
   \overbrace{\operatorname{id}_V \otimes \cdots \otimes \operatorname{id}_V}^{n-r-1}. 
$$
Here we define $t \in \operatorname{End}(V \otimes V)$ by
\begin{align*}
   t e_i e_j &= 
   \begin{cases}
   q e_j e_i, & i=j, \\
   e_j e_i, & i>j, \\
   e_j e_i + (q-q^{-1}) e_i e_j, & i<j,
   \end{cases}
\end{align*}
where $e_1,\ldots,e_n$ mean the standard basis of $V$.
Note that we omit the symbol ``$\otimes$.''
Thus we have explained the definition of $\hat{T}(V)$ as a vector space.

Moreover, we consider a natural algebra structure of $\hat{T}(V)$.
Namely, for $\sigma v_1 \cdots v_k \in \hat{T}_k(V)$ and 
$\tau w_1 \cdots w_l \in \hat{T}_l(V)$, we define their product by
$$
   \sigma v_1 \cdots v_k \cdot
   \tau w_1 \cdots w_l
   = 
   \sigma \alpha^k(\tau) v_1 \cdots v_k w_1 \cdots w_l.
$$
Here $\sigma$ and $\tau$ are elements of $H_{\infty}(q)$,
and $v_1, \ldots, v_k, w_1, \ldots, w_l$ are vectors in $V$. 
Moreover $\alpha$ is the algebra endomorphism on $H_{\infty}(q)$ defined by 
\begin{align*}
   \alpha \colon 
   H_{\infty}(q) \to H_{\infty}(q), \quad 
   t_r \mapsto t_{r+1}.
\end{align*}
This multiplication is well defined.
With this multiplication, 
$\hat{T}(V)$ becomes an associative graded algebra.

\begin{remark}
   In \cite{I}, the definition of $\bar{T}(V)$ was based 
   on the \textit{left} action of $S_p$ on $V^{\otimes p}$.
   However, in this paper, 
   we defined $\hat{T}(V)$ using the \textit{right} action of $H_p(q)$ on $V^{\otimes p}$.
   Actually we can also define a similar algebra using the left action,
   but we employ our definition because this is compatible with the action of $U_q(\mathfrak{gl}(V))$
   (see Section~\ref{sec:representation of Uq(gl)}).
\end{remark}
%
%
\section{The multiplication by $v \in V$ and the derivation by $v^* \in V^*$}
\label{sec:multiplications and derivations}
%
%
In this section we define two series of fundamental operators on $\hat{T}(V)$,
namely the \textit{multiplications} by vectors in $V$ 
and the \textit{derivations} by covectors in $V^*$.

First,
let $R(\varphi)$ denote the right multiplication by $\varphi \in \hat{T}(V)$:
$$
   R(\varphi) \colon \hat{T}(V) \to \hat{T}(V), \quad
   \psi \mapsto \psi\varphi.
$$
This operator is obviously fundamental,
and the following two cases are particularly fundamental:
(i) the case that $\varphi$ is a vector in $V$,
and (ii) the case that $\varphi$ is an element of $H_{\infty}(q)$.
Indeed the other cases can be generated by these two cases.
Note that $R(v)$ for $v \in V \subset \hat{T}_1(V)$ raises the degree by one, and
$R(\sigma)$ for $\sigma \in H_{\infty}(q) = \hat{T}_0(V)$ does not change the degree.

Next, 
we define an operator $R(v^*)$ associated to a covector $v^* \in V^*$.
When $v^*$ is a member of the dual basis $e^*_1,\ldots,e^*_n$,
we define $R(e^*_i) \in \operatorname{End}_{\mathbb{C}}(\hat{T}(V))$ by
\begin{align}
   R(e^*_i) \colon 
   \hat{T}_p(V) &\to \hat{T}_{p-1}(V), 
\label{eq:definition of R(e^*_i)} \\
   \sigma v_1 \cdots v_p
   &\mapsto
   \sum_{r=1}^p
   \sigma 
   k_i^{-1}(v_1) \cdots k_i^{-1}(v_{r-1}) 
   \langle e^*_i, v_r \rangle 
   g_i(v_{r+1}) \cdots g_i(v_p).
\notag
\end{align}
Here, 
$k_i$ is the linear transformation on $V$,
and $g_i$ is the linear map from $V$ to $H_2(q) \otimes V$ defined as follows:
$$
   k_i \colon V \to V, \quad
   e_j \mapsto q^{\delta_{ij}} e_j, \qquad\qquad
   g_i \colon V \to H_2(q) \otimes V, \quad
   e_j \mapsto 
   \begin{cases}
       t_1 e_j , & i \leq j, \\
       t_1^{-1} e_j, & i > j.
   \end{cases}
$$
Based on this, 
we define $R(v^*)$ in such a way that 
$R \colon V^* \to \operatorname{End}_{\mathbb{C}}(\hat{T}(V))$ is linear.
We call this $R(v^*)$ the \textit{derivation} by $v^* \in V^*$.

For example, we have
\begin{align*}
   R(e^*_1) e_1 e_1 e_2
   &= \langle e^*_1, e_1 \rangle t_1 e_1 t_1 e_2
   + q^{-1} e_1 \langle e^*_1, e_1 \rangle t_1 e_2
   + q^{-1} e_1 q^{-1} e_1 \langle e^*_1, e_2 \rangle \\
   &= t_1 t_2 e_1 e_2
   + q^{-1} t_1 e_1 e_2.
\end{align*}

Let us check the well-definedness of the definition (\ref{eq:definition of R(e^*_i)}) of $R(e^*_i)$.
For this,
we consider a linear map $f_r \colon T_p(V) \to \hat{T}_{p-1}(V)$ defined by
$$
   f_r(v_1 \cdots v_p)
   = k_i^{-1}(v_1) \cdots k_i^{-1}(v_{r-1}) 
   \langle e^*_i, v_r \rangle 
   g_i(v_{r+1}) \cdots g_i(v_p),
$$
so that $R(e^*_i) = \sum_{r=1}^p f_r$.
For the well-definedness of (\ref{eq:definition of R(e^*_i)}),
it suffices to show that $t_1,\ldots,t_{p-1}$ commute with $\sum_{r=1}^p f_r$.
Namely we only have to show the following lemma:

\begin{lemma} \sl
   For $s = 1,\ldots,p-1$, the following hold\textnormal{:} \\
   \textnormal{(1)}
   $t_s$ commutes with $f_r$ unless $r= s$, $s+1$. \\
   \textnormal{(2)} 
   $t_s$ commutes with $f_s + f_{s+1}$.
\end{lemma}

\begin{proof}
We put $e_J = e_{j_1} \cdots e_{j_p}$ for $J = (j_1,\ldots,j_p)$.
Let us fix $I = (i_1,\ldots,i_p)$ and $1 \leq s \leq p$,
and put 
$$
   \gamma = 
   \begin{cases}
      1, & i_s \geq i_{s+1}, \\ 
      -1, & i_s < i_{s+1},
   \end{cases}
$$
so that $t_s^{\gamma} e_I = q^{\delta_{i^{}_s i_{s+1}}} e_{I'}$
with $I' = (i_1,\ldots,i_{s-1},i_{s+1},i_s,i_{s+2},\ldots,i_p)$.
To show (1),
it suffices to show $t_s^{\gamma} f_r(e_I) = q^{\delta_{i^{}_s i_{s+1}}} f_r(e_I')$ for $r \ne s, s+1$.
When $r > s+1$,
this can be deduced from the relation 
$t_2^{\gamma} t_1^{\varepsilon} t_2^{\delta} t_1^{-\gamma}
= t_1^{\delta} t_2^{\varepsilon}$
for $\gamma, \delta,\varepsilon \in \{ 1,-1 \}$.
We can show the the case $r < s$ by a direct calculation.

We can also show (2) by a direct calculation.
\end{proof}

\begin{remark}
   This well-definedness means that $R(v^*)$ commutes with the action of $H_{\infty}(q)$.
\end{remark}
%
%
\section{Commutation relations}
%
%
For the multiplications and derivations introduced in the previous section,
we have the following commutation relations.

\begin{theorem}\label{thm:commutation relations 1}\sl
   For $i<j$, we have
   \begin{align*}
      R(e_i) R(e_i) &= q^{-1} R(e_i) R(e_i) R(t_1) = q R(e_i) R(e_i) R(t_1^{-1}), \\
      R(e_i) R(e_j) &= R(e_j) R(e_i) R(t_1^{-1}), \\
      R(e_j) R(e_i) &= R(e_i) R(e_j) R(t_1),\allowdisplaybreaks \\[10pt]
      R(e^*_i) R(e^*_i) &= q^{-1} R(t_1) R(e^*_i) R(e^*_i) = q R(t_1^{-1}) R(e^*_i) R(e^*_i), \\ 
      R(e^*_i) R(e^*_j) &= R(t_1^{-1}) R(e^*_j) R(e^*_i), \\
      R(e^*_j) R(e^*_i) &= R(t_1) R(e^*_i) R(e^*_j),\allowdisplaybreaks \\[10pt]
      R(e^*_i) R(e_i) &= R(e_i) R(t_1) R(e^*_i) + K_i^{-1} = R(e_i) R(t_1^{-1}) R(e^*_i) + K_i, \\
      R(e^*_j) R(e_i) &= R(e_i) R(t_1^{-1}) R(e^*_j), \\
      R(e^*_i) R(e_j) &= R(e_j) R(t_1) R(e^*_i),
   \end{align*}
   where $K_i$ is the linear transformation on $\hat{T}(V)$ defined by
   $$
      K_i \colon \sigma e_{i_1} \cdots e_{i_p} \mapsto 
      q^{\delta_{ii_1} + \cdots + \delta_{ii_p}} \sigma e_{i_1} \cdots e_{i_p}.
   $$
\end{theorem}

Namely we can exchange two multiplications by vectors putting 
$t_1$ or $t_1^{-1} \in H_2(q)$ on the right of these two operators.
Similarly we can exchange two derivations
putting $t_1$ or $t_1^{-1}$ on the left of two operators.
The most interesting one is the commutation relation between a derivation and a multiplication.
This time, $t_1$ or $t_1^{-1}$ appears in the middle of these two operators.
We can regard these relations as an analogue of the canonical commutation relations.

\begin{proof}[Proof of Theorem~\textsl{\ref{thm:commutation relations 1}}]
These relations can be checked by direct calculations
except for the commutation relations between two derivations. 
Thus, we here prove the sixth relation, from which
the fifth relation is immediate.
We can prove the fourth relation similarly (actually more easily).

To show the sixth relation, it suffices to prove 
$$
   R(e^*_j) R(e^*_i) e_{k_1} \cdots e_{k_m} e_i^a e_j^b
   = R(t_1) R(e^*_i) R(e^*_j) e_{k_1} \cdots e_{k_m} e_i^a e_j^b
$$ 
for $k_1,\ldots,k_m \ne i,j$,
because the derivations commute with the action of $H_{\infty}(q)$.
By the definition of derivations, we have
\begin{align*}
   R(e^*_i) e_{k_1} \cdots e_{k_m} e_i^a e_j^b 
   &= \sum_{r=1}^a e_{k_1} \cdots e_{k_m}  e_i^{r-1} (t_1 e_i)^{a-r} (t_1 e_j)^b
   = \sum_{r=1}^a e_{k_1} \cdots e_{k_m} t_r^{(a+b-r)} e_i^{a-1} e_j^b, \\
   R(e^*_j) e_{k_1} \cdots e_{k_m} e_i^a e_j^b 
   &= \sum_{s=1}^b e_{k_1} \cdots e_{k_m} e_i^a e_j^{s-1} (t_1 e_j)^{b-s}
   = \sum_{s=1}^b e_{k_1} \cdots e_{k_m} t_{a+s}^{(b-s)} e_i^a e_j^{b-1}, 
\end{align*}
where we put $t_k^{(c)} = t_k t_{k+1} \cdots t_{k+c-1}$. 
Using these, we have
\begin{align*}
   R(e^*_j) R(e^*_i) e_{k_1} \cdots e_{k_m} e_i^a e_j^b 
   &= \sum_{r=1}^a \sum_{s=1}^b e_{k_1} e_{k_1} \cdots e_{k_m}  t_r^{(a+b-r)} t_{a-1+s}^{(b-s)} e_i^{a-1} e_j^{b-1}, \\
   R(t_1)R(e^*_i)R(e^*_j) e_{k_1} \cdots e_{k_m} e_i^a e_j^b 
   &= \sum_{r=1}^a \sum_{s=1}^b e_{k_1} \cdots e_{k_m} t_{a+s}^{(b-s)} t_r^{(a+b-r)} e_i^{a-1} e_j^{b-1}.
\end{align*}
These are equal,
because we have $t_r^{(a+b-r)} t_{a-1+s}^{(b-s)} = t_{a+s}^{(b-s)} t_r^{(a+b-r)}$
by a calculation.
\end{proof}

It is natural to consider the operator algebra generated by
$R(v)$, $R(v^*)$ and $R(\sigma)$ with $v \in V$, $v^* \in V^*$ and $\sigma \in H_{\infty}(q)$.
We can regard this operator algebra as an analogue of the Weyl algebras and the Clifford algebras.

The following commutation relations with $K_i$ are also fundamental:

\begin{theorem}\label{thm:commutation relations 2}\sl
   We have
   $$
      K_j R(e_i) = q^{\delta_{ij}} R(e_i) K_j, \qquad
      K_j R(e^*_i) = q^{-\delta_{ij}} R(e^*_i) K_j, \qquad
      K_j R(t_r) = R(t_r) K_j.
   $$
\end{theorem}
%
%
\section{The natural representation of $U_q(\mathfrak{gl}(V))$ on $V^{\otimes p}$}
\label{sec:representation of Uq(gl)}
%
%
We can use the operators introduced in Section~\ref{sec:multiplications and derivations}
to study the natural representation of the quantum enveloping algebra $U_q(\mathfrak{gl}(V))$
on $V^{\otimes p}$.

First, let us recall the definition of $U_q(\mathfrak{gl}(V))$.
For $V = \mathbb{C}^n$,
we define the $\mathbb{C}$-algebra $U_q(\mathfrak{gl}(V))$
by the following generators and relations (\cite{J}):
\begin{align*}
   \text{generators: }&
   q^{\pm\varepsilon_1/2},\ldots,q^{\pm\varepsilon_n/2},
   \hat{e}_1,\ldots,\hat{e}_{n-1},
   \hat{f}_1,\ldots,\hat{f}_{n-1}, \displaybreak[0]\\
   \text{relations: }
   & q^{\varepsilon_i/2} q^{\varepsilon_j/2} = q^{\varepsilon_j/2} q^{\varepsilon_i/2}, \qquad
   q^{\varepsilon_i/2} q^{-\varepsilon_i/2} = q^{-\varepsilon_i/2} q^{\varepsilon_i/2} = 1, \\
   & q^{\varepsilon_i/2} \hat{e}_j q^{-\varepsilon_i/2} = q^{\delta_{ij} - \delta_{i,j+1}} \hat{e}_j, \qquad
   q^{\varepsilon_i/2} \hat{f}_j q^{-\varepsilon_i/2} = q^{-\delta_{ij} + \delta_{i,j+1}} \hat{f}_j, \\
   & \hat{e}_i \hat{f}_j - \hat{f}_j \hat{e}_i 
   = \delta_{ij} 
   \frac{q^{(\varepsilon_i - \varepsilon_{i+1})/2} - q^{(-\varepsilon_i + \varepsilon_{i+1})/2}}{q-q^{-1}}, \\ 
   & \hat{e}_i \hat{e}_j =  \hat{e}_j \hat{e}_i, \qquad 
   \hat{f}_i \hat{f}_j =  \hat{f}_j \hat{f}_i
   \text{ for $|i-j| > 1$}, \\
   & \hat{e}_i^2 \hat{e}_{i \pm 1} 
   - (q + q^{-1}) \hat{e}_i \hat{e}_{i \pm 1} \hat{e}_i 
   + \hat{e}_{i \pm 1} \hat{e}_i^2 = 0, \\
   & \hat{f}_i^2 \hat{f}_{i \pm 1} 
   - (q + q^{-1}) \hat{f}_i \hat{f}_{i \pm 1} \hat{f}_i 
   + \hat{f}_{i \pm 1} \hat{f}_i^2 = 0.
\end{align*}
Here we denote $q^{a_1} \cdots q^{a_k}$ simply 
by $q^{a_1 + \cdots + a_k}$.

Next, we define 
$\hat{E}_{ij}$ and  $\hat{E}_{ji} \in U_q(\mathfrak{gl}(V))$ for $1 \leq i<j \leq n$ 
by
$$
   \hat{E}_{i,i+1} = \hat{e}_i, \qquad
   \hat{E}_{i+1,i} = \hat{f}_i
$$
and recursive relations
$$
   \hat{E}_{ik} = \hat{E}_{ij} \hat{E}_{jk} - q \hat{E}_{jk} \hat{E}_{ij}, \qquad
   \hat{E}_{ki} = \hat{E}_{kj} \hat{E}_{ji} - q^{-1} \hat{E}_{ji} \hat{E}_{kj}
$$
for $i<j<k$.
Moreover, for $i<j$ and $a \in \mathbb{C} \smallsetminus \{ 0 \}$, we put
$$
   \hat{E}_{ij}(a) =a^{-1} q^{-(\varepsilon_i + \varepsilon_j -1)/2} \hat{E}_{ij},\quad
   \hat{E}_{ji}(a) =a q^{(\varepsilon_j + \varepsilon_i -1)/2} \hat{E}_{ji},\quad
   \hat{E}_{ii}(a) = \frac{a q^{\varepsilon_i} - a^{-1} q^{-\varepsilon_i}}{q-q^{-1}}.
$$
We call this $\hat{E}_{ij}(a)$ for $1 \leq i,j \leq n$ the $L$-operator.

We denote by $\pi$ 
the natural representation of the quantum enveloping algebra $U_q(\mathfrak{gl}(V))$
on $V^{\otimes p}$.
This is determined by the following actions of generators (\cite{J}):
\begin{align*}
   \pi(\hat{e}_i) 
   &= \sum_{r=1}^p 
   k_i^{1/2} k_{i+1}^{-1/2} \otimes \cdots \otimes k_i^{1/2} k_{i+1}^{-1/2} 
   \otimes \underbrace{E_{i,i+1}}_{\text{$r$th}} 
   \otimes k_i^{-1/2} k_{i+1}^{1/2} \otimes \cdots \otimes k_i^{-1/2} k_{i+1}^{1/2}, \\
   \pi(\hat{f}_i) 
   &= \sum_{r=1}^p
   k_i^{1/2} k_{i+1}^{-1/2} \otimes \cdots \otimes k_i^{1/2} k_{i+1}^{-1/2} 
   \otimes \underbrace{E_{i+1,i}}_{\text{$r$th}} 
   \otimes k_i^{-1/2} k_{i+1}^{1/2} \otimes \cdots \otimes k_i^{-1/2} k_{i+1}^{1/2}, \\
   \pi(q^{\pm\varepsilon_i/2}) 
   &= k_i^{\pm 1/2} \otimes \cdots \otimes k_i^{\pm 1/2} = K_i^{\pm 1/2}.
\end{align*}
Here $k_i^{1/2}$ and $E_{ij}$ are the linear transformations on $V$
defined by 
$$
   k_i^{1/2} \colon e_h \mapsto q^{\delta_{ih}/2} e_h, \qquad
   E_{ij} \colon e_h \mapsto \delta_{jh} e_i.
$$

We can use our operators to express this representation $\pi$:

\begin{theorem}\label{thm:expression of pi} \sl
   We have
   $$
      \pi(\hat{E}_{ij}(1)) = R(e_i) R(e^*_j).
   $$
\end{theorem}

\begin{proof}
We can check the assertion by a direct calculation when $i=j$.

Let us show the case $i \ne j$.
We note that
$$
   \pi(\hat{E}_{ij}(1)) = q^{-1/2} K_i^{-1/2} K_j^{-1/2} \pi(\hat{E}_{ij}) , \qquad
   \pi(\hat{E}_{ji}(1)) = q^{1/2} K_j^{1/2} K_i^{1/2} \pi(\hat{E}_{ji})
$$
for $i<j$.
Thus, it suffices to show 
\begin{equation}\label{eq:key relation of the expression of pi}
   \pi(\hat{E}_{ij}) = K_j^{1/2} R(e_i) R(e^*_j) K_i^{1/2}, \qquad
   \pi(\hat{E}_{ji}) = K_i^{-1/2} R(e_j) R(e^*_i) K_j^{-1/2}
\end{equation}
for $i<j$.
We can check these relations for $j = i + 1$ by a direct calculation.
To show the other cases, we put
$$
   F_{ij} = K_j^{1/2} R(e_i) R(e^*_j) K_i^{1/2}, \qquad
   F_{ji} = K_i^{-1/2} R(e_j) R(e^*_i) K_j^{-1/2}
$$
for $i < j$.
Then, we have 
$$
   F_{ik} = F_{ij} F_{jk} - q F_{jk} F_{ij}, \qquad
   F_{ki} = F_{kj} F_{ji} - q^{-1} F_{ji} F_{kj}
$$
for $i<j<k$.
Indeed, using Theorems~\ref{thm:commutation relations 1} and~\ref{thm:commutation relations 2},
we see the first relation as follows:
\begin{align*}
   &F_{ij} F_{jk} - q F_{jk} F_{ij} \\
   &\qquad= K_j^{1/2} R(e_i) R(e^*_j) K_i^{1/2} K_k^{1/2} R(e_j) R(e^*_k) K_j^{1/2} \\
   & \qquad\qquad 
   - q K_k^{1/2} R(e_j) R(e^*_k) K_j^{1/2} K_j^{1/2} R(e_i) R(e^*_j) K_i^{1/2} 
   \allowdisplaybreaks\\
   &\qquad= K_j^{1/2} K_k^{1/2} R(e_i) R(e^*_j) R(e_j) R(e^*_k) K_i^{1/2} K_j^{1/2} \\
   & \qquad\qquad
   - K_k^{1/2} K_j^{1/2} R(e_j) R(e^*_k) R(e_i) R(e^*_j) K_j^{1/2} K_i^{1/2} 
   \allowdisplaybreaks \\
   &\qquad= K_j^{1/2} K_k^{1/2} R(e_i) (R(e_j) R(t_1) R(e^*_j) + K_j^{-1}) R(e^*_k) K_i^{1/2} K_j^{1/2} \\
   & \qquad\qquad 
   - K_k^{1/2} K_j^{1/2} R(e_j) R(e_i) R(t_1^{-1}) R(e^*_k) R(e^*_j) K_j^{1/2} K_i^{1/2} 
   \allowdisplaybreaks\\
   &\qquad= K_j^{1/2} K_k^{1/2} R(e_i) (R(e_j) R(t_1) R(e^*_j) + K_j^{-1}) R(e^*_k) K_i^{1/2} K_j^{1/2} \\
   & \qquad\qquad
   - K_k^{1/2} K_j^{1/2} R(e_i) R(e_j) R(t_1) R(e^*_j) R(e^*_k) K_j^{1/2} K_i^{1/2} 
   \allowdisplaybreaks\\
   &\qquad= K_k^{1/2} R(e_i) R(e^*_k) K_i^{1/2} \\
   &\qquad= F_{ik}.
\end{align*}
We can show the second relation similarly.
Combining these, we have (\ref{eq:key relation of the expression of pi}).
\end{proof}

\begin{remark}
   Theorem~\ref{thm:expression of pi} is quite similar 
   to the natural action of the Lie algebra $\mathfrak{gl}(V)$ on $\mathcal{P}(V)$
   the space of all polynomial functions on $V$.
   This action $\mu$ can be expressed as 
   $$
      \mu(E_{ij}) = x_i \frac{\partial}{\partial x_j}.
   $$
   Here $x_i$ means the canonical coordinate of $V$,
   and $E_{ij}$ means the standard basis of $\mathfrak{gl}(V)$.
\end{remark}

Using Theorems~\ref{thm:commutation relations 1} and~\ref{thm:commutation relations 2}, 
we have the following relations:

\begin{proposition}\label{prop:commutation relations 3} \sl
   We have
   \begin{alignat*}{2}
      R(e_i) R(e_j) R(e^*_k) &= R(e_j) R(e^*_k) R(e_i) &\qquad 
      &\text{when $i \lessgtr j$ and $i \lessgtr k$}, \allowdisplaybreaks\\
      R(e_i) R(e_j) R(e^*_k) &= R(e_j) R(e^*_k) R(e_i) &&\\
      &\qquad
      \pm (q-q^{-1}) R(e_i) R(e^*_k) R(e_j) &\qquad 
      &\text{when $j \lessgtr i \lessgtr k$}, \allowdisplaybreaks\\
      R(e_i) R(e_j) R(e^*_i) &= R(e_j) R(e^*_i) R(e_i) 
      - K_i^{\pm 1} R(e_j) &\qquad 
      &\text{when $i \lessgtr j$}, \allowdisplaybreaks\\
      R(e_i) R(e_i) R(e^*_j) &= q^{\pm 1} R(e_i) R(e^*_j) R(e_i) 
      &\qquad &\text{when $i \lessgtr j$}, \allowdisplaybreaks\\
      R(e_i) R(e_i) R(e^*_i) 
      &= q R(e_i) R(e^*_i) R(e_i) - K_i R(e_i) && \\
      &= q^{-1} R(e_i) R(e^*_i) R(e_i) - K_i^{-1} R(e_i).&&
   \end{alignat*}
\end{proposition}

Moreover we have the following proposition.
Indeed, using Proposition~\ref{prop:commutation relations 3},
we can rewrite $R(v_k) \cdots R(v_1) R(v^*_1) \cdots R(v^*_k)$
as a sum of products of $R(v)R(v^*)$ and $K_i$.

\begin{proposition}\label{prop:R(v)^kR(v^*)^k} \sl
   For any $v_1,\ldots,v_k \in V$ and 
   $v^*_1,\ldots,v^*_k \in V^*$, 
   we have
   $$
      R(v_k) \cdots R(v_1) R(v^*_1) \cdots R(v^*_k) \in \pi(U_q(\mathfrak{gl}(V))).
   $$
\end{proposition}
%
%
\section{$q$-Schur--Weyl duality}
%
%
We can use our results to prove the following Jimbo duality,
namely the $q$-analogue of the Schur--Weyl duality.
This theorem was first given in \cite{J}, and several proofs have been given (see \cite{H}, \cite{Z} for example).

\begin{theorem}
\label{thm:q-Schur--Weyl duality}\sl
   Assume that $[p]! \ne 0$.
   Let us denote by $\rho$ the natural action of $H_p(q)$ on $V^{\otimes p}$.
   Then $\rho(H_p(q))$ and $\pi(U_q(\mathfrak{gl}(V)))$ are mutual commutants of each other.
   Namely we have
   $$
      \operatorname{End}(V^{\otimes p})^{\rho(H_p(q))} = \pi(U_q(\mathfrak{gl}(V))), \qquad
      \operatorname{End}(V^{\otimes p})^{\pi(U_q(\mathfrak{gl}(V)))} = \rho(H_p(q)).
   $$
\end{theorem}

Here $[k] = [k]_q$ is a $q$-integer, and $[k]! = [k]_q!$ is a $q$-factorial:
$$
   [k] 
   = \frac{q^k-q^{-k}}{q-q^{-1}}
   = q^{k-1} + q^{k-3} + \cdots + q^{-k+1}, \qquad
   [k]! = [k] [k-1] \cdots [1].
$$

To prove this theorem, we consider the following analogue of the Euler operator:
$$
   \mathcal{E} = \sum_{J \in \mathcal{J}} \frac{1}{J!}
   R(e_{j_1}) \cdots R(e_{j_p}) R(e^*_{j_p}) \cdots R(e^*_{j_1}).
$$
Here we put
$$
   \mathcal{J} = \{ (j_1,\ldots,j_p) \in \mathbb{N}^p \,|\, 1 \leq j_1 \leq \cdots \leq j_p \leq n \}.
$$
Moreover we put $J! = [m_1]! \cdots [m_n]!$ for $J = (j_1,\ldots,j_p) \in \mathcal{J}$,
where $m_i$ is the multiplicity of $j_1,\ldots,j_p$ at $i$\textnormal{:}
$$
   (j_1,\ldots,j_p) 
   = (
   \underbrace{1,\ldots,1}_{m_1},
   \underbrace{2,\ldots,2}_{m_2},\ldots,
   \underbrace{n,\ldots,n}_{m_n}).
$$ 
For this $\mathcal{E}$, the following relation holds:

\begin{lemma}\label{lemma:Euler type operator}\sl
   We have
   $\mathcal{E} \varphi = \varphi$ for any $\varphi \in V^{\otimes p}$.
\end{lemma}

\begin{proof}
We put 
$$
   \Phi_i(a)= \frac{q^aK_i - q^{-a}K_i^{-1}}{q - q^{-1}},
$$
and moreover
$$
   \Phi_i^{(m)} 
   = \Phi_i(0) \Phi_i(-1) \cdots \Phi_i(-m+2) \Phi_i(-m+1).
$$
Then we have $R(e_i)R(e^*_i) = \Phi_i(0)$
and $\Phi_i(a)R(e^*_i) =  R(e^*_i)\Phi_i(a-1)$,
so that
$$
   R(e_i)^m R(e^*_i)^m 
   = \Phi_i^{(m)}.
$$
Thus, for $1 \leq j_1 \leq \cdots \leq j_p \leq n$, we have
$$
   R(e_{j_1}) \cdots R(e_{j_p}) R(e^*_{j_p}) \cdots R(e^*_{j_1}) 
   = \Phi_{j_1}^{(m_1)}
   \Phi_{j_2}^{(m_2)}
   \cdots
   \Phi_{j_{p-1}}^{(m_{p-1})}
   \Phi_{j_p}^{(m_p)},
$$
where $m_i$ is the multiplicity of $j_1,\ldots,j_p$ at $i$.

Moreover, we consider 
$1 \leq i_1,\ldots,i_p \leq n$,
and let $l_i$ be the multiplicity of $i_1,\ldots,i_p$ at $i$.
Then we have
$$
   \Phi_i(a) e_{i_1} \cdots e_{i_p}
   = [l_i + a] e_{i_1} \cdots e_{i_p}.
$$
Hence, we have
\begin{align*}
   &R(e_{j_1}) \cdots R(e_{j_p}) R(e^*_{j_p}) \cdots R(e^*_{j_1}) e_{i_1} \cdots e_{i_p} \\
   & \qquad
   = [l_1]^{(m_1)} \cdots [l_n]^{(m_n)} e_{i_1} \cdots e_{i_p}\\
   & \qquad
   = \begin{cases}
   [m_1]! \cdots [m_n]! e_{i_1} \cdots e_{i_p}, &
   (l_1,\ldots,l_n) = (m_1,\ldots,m_n), \\
   0, &
   (l_1,\ldots,l_n) \ne (m_1,\ldots,m_n).
   \end{cases}
\end{align*}
Here we put $[l]^{(m)} = [l][l-1]\cdots[l-m+1]$.
The assertion is immediate from this.
\end{proof}

Using this lemma, we can prove Theorem~\ref{thm:q-Schur--Weyl duality} as follows:

\begin{proof}[Proof of Theorem~\textsl{\ref{thm:q-Schur--Weyl duality}}]
We can check by a direct calculation that these two actions are commutative.
When $[p]! \ne 0$, the algebra $H_p(q)$ is semisimple (\cite{GU}).
Thus, by the double commutant theorem (\cite{GW}), 
it suffices to show 
$\operatorname{End}(V^{\otimes p})^{\rho(H_p(q))} \subset \pi(U_q(\mathfrak{gl}(V)))$.

Assume that $f \in\operatorname{End}(V^{\otimes p})^{\rho(H_p(q))}$.
Then, for any $\varphi \in V^{\otimes p}$, we have
\begin{align*}
   f(\varphi)
   & = f(\mathcal{E}\varphi) \allowdisplaybreaks\\
   & = f(\sum_{J = (j_1,\ldots,j_p) \in \mathcal{J}} \frac{1}{[J]!}
   R(e_{j_p}) \cdots R(e_{j_1}) R(e^*_{j_1}) \cdots R(e^*_{j_p}) \varphi) \allowdisplaybreaks\\
   & = f(\sum_{J = (j_1,\ldots,j_p) \in \mathcal{J}} \frac{1}{[J]!}
   R(e_{j_p}) \cdots R(e_{j_1}) \sigma_J) \allowdisplaybreaks\\
   & = f(\sum_{J = (j_1,\ldots,j_p) \in \mathcal{J}} \frac{1}{[J]!}
   \sigma_J e_{j_1} \cdots e_{j_p}) \allowdisplaybreaks\\
   & = \sum_{J = (j_1,\ldots,j_p) \in \mathcal{J}} \frac{1}{[J]!}
   \sigma_J f(e_{j_1} \cdots e_{j_p}) \allowdisplaybreaks\\
   & = \sum_{J = (j_1,\ldots,j_p) \in \mathcal{J}} \frac{1}{[J]!}
   R(f(e_{j_1} \cdots e_{j_p}))  R(e^*_{j_1}) \cdots R(e^*_{j_p}) \varphi.   
\end{align*}
Here we denote $R(e^*_{j_1}) \cdots R(e^*_{j_p}) \varphi$ simply by $\sigma_J$.
This $\sigma_J$ is an element of $H_p(q)$,
and $f$ commutes with the action of $H_p(q)$,
so that the fifth equality holds.

Thus, by Proposition~\ref{prop:R(v)^kR(v^*)^k}, 
we see that $f \in \pi(U_q(\mathfrak{gl}(V)))$.
\end{proof}

\begin{remark}
   For any group $G$, every map $f \colon G \to G$ commuting 
   with all right translations is equal to a left translation.
   This fact is proved quickly as follows.
   Let $e$ be the identity element of $G$.
   For any element $x$ of $G$, we have $f(x) = f(ex) = f(e)x$,
   because $f$ commutes with the right multiplication by $x$.
   Thus $f$ is equal to the left multiplication by $f(e)$, as we claimed.
   It should be noted that
   our proof of Theorem~\ref{thm:q-Schur--Weyl duality}
   is based on the same principle
   (the operator $\mathcal{E}$ plays a role of the identity element $e$).
\end{remark}

Theorem~\ref{thm:q-Schur--Weyl duality} holds, if and only if $q$ satisfy $[p]! \ne 0$
(this condition is also equivalent with the condition that $H_p(q)$ is semisimple).
Indeed, when $[p]! = 0$, this proof fails because there exists $I$ such that $[I]! = 0$.
It is interesting that the condition $[p]! = 0$ appears this way.

I hope that the algebra $\hat{T}(V)$ and the differential operators on $\hat{T}(V)$
will be useful to study invariant theory in quantum enveloping algebras.

%
%
%


\begin{thebibliography}{GW}

\bibitem[GW]{GW}
R. Goodman and N. R. Wallach, 
Representations and invariants of the classical groups,
Cambridge Univ.\ Press, 2003.

\bibitem[GU]{GU}
A. Gyoja and K. Uno,
{\it On the semisimplicity of Hecke algebras},
J.\ Math.\ Soc.\ Japan, {\bf 41} (1989), no.~1, 75--79.

\bibitem[H]{H}
T. Hayashi, 
{\it Quantum deformation of classical groups},
Publ.\ Res.\ Inst.\ Math.\ Sci.\ {\bf 28} (1992), no.~1, 57--81.

\bibitem[I]{I}
M. Itoh,
{\it Extensions of the tensor algebra and their applications},
Comm.\ Algebra {\bf 40} (2012), no.~9, 3442--3493.

\bibitem[J]{J}
M. Jimbo,
{\it A $q$-analogue of $U(\mathfrak{gl}(N+1))$, Hecke algebra, and the Yang--Baxter equation},
Lett.\ Math.\ Phys.\ {\bf 11} (1986), no.~3, 247--252.

\bibitem[O1]{O1}
A. Okounkov,
{\it Quantum immanants and higher Capelli identities}, 
Transform.\ Groups {\bf 1} (1996), no.~1, 99--126.

\bibitem[O2]{O2}
\bysame,
{\it Young basis, Wick formula, and higher Capelli identities},
Internat.\ Math.\ Res.\ Notices 1996, no.~17, 817--839. 

\bibitem[Z]{Z}
R. B. Zhang,
{\it Howe duality and the quantum general linear group},
Proc.\ Amer.\ Math.\ Soc.\ {\bf 131} (2003), no.~9, 2681--2692


\end{thebibliography}
\end{document}